\documentclass[12pt]{amsart}
\usepackage[english]{babel}
\usepackage{amssymb}
\usepackage{enumerate}
\usepackage{amsthm}
\usepackage{amsmath}

\usepackage[margin=1in]{geometry}
\usepackage{setspace}

\usepackage{hyperref}

\newtheorem{theorem}{Theorem}

\newtheorem{lemma}[theorem]{Lemma}
\newtheorem{definition}[theorem]{Definition}
\theoremstyle{definition}
\newtheorem{rmk}[theorem]{Remark}

\theoremstyle{remark}

\newcommand{\Ric}{\operatorname{Ric}}
\DeclareMathOperator{\area}{area}

\newcommand{\R}{\mathbb{R}}

\newcommand{\cH}{\mathcal{H}}
\newcommand{\cI}{\mathcal{I}}
\newcommand{\cL}{\mathcal{L}}

\title{Isoperimetric structure of asymptotically conical manifolds}
\author{Otis Chodosh}
\address{Department of Mathematics, Stanford University, CA-94305, United States of America}
\email{ochodosh@math.stanford.edu}
\author{Michael Eichmair}
\address{Department of Mathematics, ETH, 8092 Z\"urich, Switzerland}
\email {michael.eichmair@math.ethz.ch}
\author{Alexander Volkmann}
\address{Max Planck Institute for Gravitational Physics, 14476 Potsdam-Golm, Germany}
\email{alexander.volkmann@aei.mpg.de}

\begin{document}
\begin{abstract}
We study the isoperimetric structure of Riemannian manifolds that are asymptotic to cones with non-negative Ricci curvature. Specifically, we generalize to this setting the seminal results of G. Huisken and S.-T. Yau \cite{Huisken-Yau:1996} on the existence of a canonical foliation by volume preserving stable constant mean curvature surfaces at infinity of asymptotically flat manifolds as well as the results of the second-named author with S. Brendle \cite{offcenter} and J. Metzger \cite{isostructure, hdiso} on the isoperimetric structure of asymptotically flat manifolds. We also include an observation on the isoperimetric cone angle of such manifolds. This result is a natural analogue of the positive mass theorem in this setting. 
\end{abstract}

\maketitle


\section{Introduction}

The seminal results of G. Huisken and S.-T. Yau \cite{Huisken-Yau:1996} on the existence of a canonical foliation at infinity of asymptotically flat manifolds and the uniqueness of the leaves of this foliation (which has been refined by J. Qing and G. Tian \cite{Qing-Tian:2007}) has had considerable repercussions on the study of such manifolds, both from a physical and a geometric point of view. We refer the reader to the introductions of \cite{Huang:2010, stablePMT, isostructure, hdiso} for recent accounts on these developments. The goal of this paper is to extend many of these results to asymptotically conical manifolds. \\

Let $m \geq 2$ be an integer and $(L, g_L)$ be a connected closed Riemannian manifold of dimension $m-1$. If $(M, g)$ is an $m$ dimensional Riemannian manifold such that there exists a compact set $K \subset M$ and a diffeomorphism 
\[
M \setminus K \cong (1, \infty) \times L
\]
with
\begin{align} \label{expansionmetric}
g = dr^2 + r^2 g_L + o(1) \qquad \text{ as } r \to \infty,
\end{align}
then we say that $(M, g)$ is asymptotically conical with link $(L, g_L)$. 

If $k \geq 0$ is an integer, we say that the expansion  (\ref{expansionmetric}) holds in $C^{k}$ if 
\[
\sum_{\ell = 0}^k  r^\ell |\nabla^\ell (g - g_C)| = o (1) \qquad \text{ as } r \to \infty. 
\]
Here, $\nabla$ is the Levi-Civita connection of the cone $(C, g_C)$ where 
\[
C = (0, \infty) \times L  \qquad \text{ and }  \qquad g_C = dr^2 + r^2 g_L
\]
and norms are computed with respect to $g_C$. Similarly, given $\alpha \in (0, 1)$, we say the expansion holds in $C^{k, \alpha}$ if the weighted $C^{k, \alpha}$ norms of $g - g_C$ tend to zero as $r \to \infty$.

Given $r > 1$ we denote by $B_r$ the region in $M$ consisting of $K$ and the diffeomorphic image of the set $(1, r) \times L$. It is convenient to introduce a smooth positive function $|\cdot| : M \to \R$ that extends the coordinate function $r$ on $M \setminus B_{2r}$ to all of $M$. \\


In this paper, we study the isoperimetric structure of asymptotically conical manifolds $(M, g)$ of dimension $m$ whose link $(L, g_L)$ satisfies conditions
\begin{align}
\Ric_L &\geq (m -2) g_L \label{Riccicondition}  \\
\text{area} (L, g_L) &< \omega_{m-1} \label{areacondition}
\end{align}
where $\omega_{m-1}$ is the volume of the unit sphere. In view of (\ref{Riccicondition}) and R. Bishop's theorem, (\ref{areacondition}) is equivalent to the requirement that $(L, g_L)$ is not isometric to the $(m-1)$-dimensional unit sphere. Our results show that conditions \eqref{Riccicondition} and \eqref{areacondition} are in many ways analogous to positivity of mass and non-negativity of scalar curvature for asymptotically flat manifolds. \\


The existence of a foliation through volume preserving stable CMC surfaces asserted in our first result below comes from an application of the inverse function theorem, cf. \cite{Ye:1996b}. An analogous result in the asymptotically flat setting was first proven by G.\ Huisken and S.-T.\ Yau in \cite{Huisken-Yau:1996} using volume preserving mean curvature flow.

\begin{theorem} [Canonical foliation by CMC surfaces] \label{existencefoliation}
Let $(M, g)$ be an asymptotically conical manifold of dimension $m$ such that the expansion (\ref{expansionmetric}) holds in $C^{2, \alpha}$ and with 
\[
\lambda_1 (- \Delta_L) > (m-1).
\]
There are $\delta > 0$ and $V_0 > 0$ such that the following hold. Let $V > V_0$. Let $r > 1$ be such that $\mathcal{L}^m (B_r) = V$. There is exactly one $u_V \in C^{2, \alpha} (L)$ with 
\[
- \delta \leq u_V \leq \delta
\]
such that 
\[
\Sigma_V = \{ (r(1 + u_V (x)), x ) : x \in L\}
\] 
is a CMC surface that encloses volume $V$. Moreover, $\Sigma_V$ is volume preserving stable, the surfaces 
\[
\{\Sigma_V\}_{V > V_0}
\]
form a foliation of the complement of a compact subset of $M$, and the weighted $C^{2, \alpha}$ norms of $u_V$ tend to zero as $V \to \infty$.  
\end{theorem}

\begin{rmk} If $(L, g_L)$ satisfies conditions (\ref{Riccicondition}) and (\ref{areacondition}) then 
\[
\lambda_1 (- \Delta_L) > m-1
\] 
by the Lichnerowicz eigenvalue estimate \cite[Theorem 9 of Chapter III]{Chavel:1984}.
\end{rmk}


\begin{theorem} [Existence and uniqueness of large isoperimetric regions] \label{thm:uniquenessiso}
Let $(M, g)$ be an asymptotically conical Riemannian manifold of dimension $m$ whose link $(L, g_L)$ satisfies conditions \eqref{Riccicondition} and \eqref{areacondition} and such that \eqref{expansionmetric} holds in $C^{1, \alpha}$. There is $V_0 >0$ with the following property. For every $V > V_0$ there is an isoperimetric region of volume $V$. Every such region $\Omega_V$ is regular and close to $B_r$ where $r>1$ is such that 
\[
\mathcal{L}^m_g(B_r) = V.
\] 
If the expansion (\ref{expansionmetric}) holds in $C^{2, \alpha}$, then 
\[
\partial \Omega_V = \Sigma_V
\]
where $\Sigma_V$ is as in Theorem \ref{existencefoliation}. In particular, $\Omega_V$ is the unique isoperimetric region of volume $V$. 
\end{theorem}

\begin{rmk} In Appendix \ref{app:alternativecondition} we show that condition \eqref{Riccicondition} in Theorem \ref{thm:uniquenessiso} may be replaced by a weaker purely isoperimetric condition on the link $(L, g_L)$.
\end{rmk}

An analogous result in the asymptotically flat setting has been shown by the second-named author and J.\ Metzger \cite{isostructure,hdiso}, building on work of H.\ Bray \cite{Bray:1997} who characterized isoperimetric regions in the Schwarzschild metric. Subsequently, the first-named author \cite{Chodosh:2014} established related results for asymptotically hyperbolic three manifolds. The proof of Theorem \ref{thm:uniquenessiso} is modeled on these results and relies on the work of F. Morgan and M. Ritor\'e \cite{Morgan-Ritore:2002} who have classified the isoperimetric regions of cones whose link satisfies \eqref{Riccicondition} and \eqref{areacondition}.


\begin{definition} Let $(M, g)$ be asymptotically conical. Given a subset $S \subset M$ we let 
\[
\underline {r} (S) =  \sup \{ r > 1 : B_r \cap S = \emptyset\}
\]
and 
\[
\overline {r} (S) = \inf \{ r > 1 : S \subset B_r\}. 
\]
\end{definition}


The following theorem is the analogue for asymptotically conical manifolds of the seminal uniqueness results for volume preserving stable constant mean curvature surfaces in asymptotically flat manifolds due G.\ Huisken and S.-T.\ Yau \cite{Huisken-Yau:1996} (and refined by J.\ Qing and G.\ Tian \cite{Qing-Tian:2007}) as well as the results of the second-named author with S. Brendle \cite{offcenter} and J.\ Metzger \cite{stablePMT}. 

\begin{theorem} [Uniqueness of large closed volume preserving stable CMC surfaces] \label{thm:uniquestable} 
Let $(M, g)$ be a $3$-dimensional asymptotically conical Riemannian manifold with link $(L, g_L)$ such that the expansion (\ref{expansionmetric}) holds in $C^{2, \alpha}$ and such that the scalar curvature $R$ of $(M, g)$ satisfies the estimate
\[
R \geq - O (r^{-2 - \epsilon})
\] 
for some $\epsilon > 0$. We also assume that 
\[
K_L \geq 1
\] 
with strict inequality on a dense set, where $K_L$ is the Gaussian curvature of the link. There are $\alpha, \beta > 0$ with the following property. Let $\Sigma$ be a connected closed volume preserving stable CMC surface in $(M, g)$ with $\underline {r} (\Sigma) > \alpha$. If 
\[
\overline {r} (\Sigma) H(\Sigma) > \beta,
\]
then $\Sigma$ is close to an outlying geodesic sphere of radius $2 / H(\Sigma)$. If 
\[
\overline {r} (\Sigma) H(\Sigma) \leq \beta,
\]
then $\Sigma$ is part of the canonical foliation $\{\Sigma_V\}_{V > V_0}$ of Theorem \ref{existencefoliation}. 
\end{theorem}

\begin{rmk} The assumption on the Gaussian curvature of the link here is equivalent to requiring that the exact cone $(C, g_C)$ contains no flat regions. Such an assumption is clearly necessary. The a priori assumption that $\underline{r} (\Sigma) \geq \alpha$ may be relaxed to the assumption 
\[
\overline{r} (\Sigma) \geq \alpha
\]
if $(M, g)$ has non-negative scalar curvature. This is because the estimate \eqref{auxhestimate} in the proof of Theorem \ref{thm:uniquestable} follows directly from Lemma \ref{lem:CY} in this case. If the scalar curvature of $(M, g)$ is positive, the condition $\underline{r} (\Sigma) \geq \alpha$ may be relaxed to the assumption 
\[
\text{area}(\Sigma) \geq \alpha. 
\]
This follows from the monotonicity formula and \eqref{CY}.
The theorem also applies to immersed surfaces if in the latter alternative of the conclusion we allow for the possibility that the immersion is a cover of the leaf.  
\end{rmk}
\begin{rmk}
The classical Alexandrov theorem classifying closed embedded CMC hypersurfaces in space forms (with no stability assumption) has recently been extended to a wide class of rotationally symmetric metrics by S.\ Brendle \cite{Brendle:IHES}. 
\end{rmk}

\begin{theorem}[Non-existence of unbounded complete stable minimal immersions]\label{thm:min-surf}
Let $(M,g)$ be an asymptotically conical $3$-manifold such that the expansion \eqref{expansionmetric} holds in $C^2$ and whose link satisfies $K_L>1$. There is no complete stable minimal immersion into $(M, g)$ that has unbounded trace. 
\end{theorem}

\begin{rmk}
In fact, we can rule out minimal immersions with finite index and unbounded trace. The reason is that such an immersion is stable outside a compact set. The proof goes through without change.  
\end{rmk}

Building on the ideas of R.\ Schoen and S.-T.\ Yau in their proof of the positive mass theorem \cite{Schoen-Yau:1979-pmt1}, analogous results in the asymptotically flat setting have been obtained by the second-named author and J.\ Metzger \cite{stablePMT} and by A.\ Carlotto \cite{Carlotto:min-surf}. The assumption that $K_{L}>1$ is quite natural in comparison with these works, as they all require that either the scalar curvature is strictly positive or that the mass is --- in a certain sense --- distributed uniformly at infinity. It is remarkable that no further condition such as a sign on the scalar curvature of $(M, g)$ is required in Theorem \ref{thm:min-surf}. This is a strong contrast with existing results in the asymptotically flat setting.

Our final result on the isoperimetric structure of asymptotically conical manifolds here is an observation related to G. Huisken's concept of the isoperimetric cone angle. 

\begin{definition}[G. Huisken \cite{Huisken:cone-video,Huisken:MM-iso-mass-video}] The isoperimetric cone angle $c_{iso}(M,g)$ of a non-compact $m$-dimensional Riemannian manifold $(M,g)$ is defined as
\begin{align} \label{def:coneangle}
\inf\left\{ \frac{\mathcal H_g^{m-1}(\partial \Omega)^m}{m^{m-1}\omega_{m-1}\, \mathcal L_g^m(\Omega)^{m-1}} : \text{$\Omega \Subset M$ open with $C^2$ outward minimizing boundary}\right\}.
\end{align}
\end{definition}

\begin{rmk}
G. Huisken's original definition differs from \eqref{def:coneangle} by a positive power depending on $m$. 
\end{rmk}

\begin{theorem}\label{thm:coneangleestimate}
Let $(M,g)$ be an asymptotically conical Riemannian manifold of dimension $m \geq 2$ with non-negative Ricci curvature. Then
\[
c_{iso}(M,g) \leq 1,
\]
with equality if and only if $(M,g)$ is isometric to Euclidean space. 
\end{theorem}

This rigidity result for the isoperimetric cone angle is in some form analogous to the positive mass theorem. 

\begin{rmk} Let $(L, g_L)$ be a closed Riemannian manifold. F. Duzaar and K. Steffen have shown that the large isoperimetric regions of the product
\[
(\mathbb R \times L, dt^2 + g_L)
\]
are slabs \cite[Proposition 2.11]{Duzaar-Steffen:1996}. M. Ritor\'e and E. Vernadakis \cite{RitoreVernadakis:RnxM} and J.G. P\'erez \cite{Perez:RnxM} have recently extended the characterization of large isoperimetric regions in Riemannian products of $(L, g_L)$ with $\R^k$ for any $k \geq 1$. The optimizers are of the form $D \times L$ where $D \subset \R^k$ is a ball. Recently, M.\ Ritor\'e and E.\ Vernadakis \cite{RitoreVernadakis:cone} have studied relative isoperimetric regions in convex regions that are asymptotically conical. 
\end{rmk}

Finally, we mention the following remarkable characterization of the isoperimetric cone angle.

\begin{theorem}[G. Huisken \cite{Huisken:cone-video,Huisken:MM-iso-mass-video}] 
Let $(M,g)$ be a complete non-compact Riemannian $3$-manifold with positive Ricci curvature. Then
\begin{align} \label{alternativeconeangle}
c_{iso}(M,g) = \inf\left\{ \frac{1}{16 \pi} \int_{\partial\Omega} H^2 : \text{$\Omega \Subset M$ open with $C^2$ outward minimizing boundary} \right\}.
\end{align}
If the infimum in either \eqref{def:coneangle} or \eqref{alternativeconeangle} is attained by $\Omega \Subset M$ open with $C^2$ outward minimizing boundary, then $(M \setminus \Omega, g)$ is isometrically contained in a cone. 
\end{theorem} 


{\bf Acknowledgments: } O. Chodosh would like to thank his advisor S. Brendle for his encouragment and support. He is grateful for the hospitality of the Institute for Mathematical Research (FIM) of ETH Z\"urich. He also acknowledges the support of the United States National Science Foundation dissertation fellowship DGE-1147470. M. Eichmair would like to thank J. Metzger and M. Ritor\'e for stimulating discussions. He is grateful for the support of the Swiss National Science Foundation grant SNF 200021-140467.   
A. Volkmann would like to thank his advisor G. Huisken for his encouragement and support. He is grateful for the hospitality of the FIM and that of S. Brendle. He also acknowledges United States National Science Foundation grant DMS-1201924 for supporting a visit to Stanford University, where some of this research was completed.


\section {Preliminary lemmas}

Let $(M, g)$ be a homogeneously regular Riemannian manifold of dimension $m$. Given $V > 0$ smaller than the volume of the manifold, we let 
\[
A(V) = \inf \{ \mathcal{H}^{m-1}_g (\partial^*  \Omega) :  \Omega \subset M \text{ a bounded open region with } \mathcal{L}^m_g ( \Omega) = V\}.
\]
A bounded open region $\Omega \subset M$ with 
\[
\mathcal {L}^m_g(\Omega) = V \qquad \text{ and } \qquad \mathcal{H}^{m-1}_g (\partial^* \Omega) = A(V)
\]
is called an isoperimetric region. The reduced boundary of an isoperimetric region is a regular hypersurface that is relatively open in its closure and whose complement in its closure has Hausdorff dimension at most $8$ in $(M, g)$.


\begin{lemma} \label{lem:minimizingsequences} Let $(M, g)$ be an asymptotically conical Riemannian manifold of dimension $m$. Let $V > 0$. There is $\rho > 0$ and an isoperimetric region $\Omega \subset M$ such that the following hold.
\begin{enumerate} [(i)]
\item $\omega_{m-1} \rho^m / m  + \mathcal{L}^m_g (\Omega) = V$
\item $\omega_{m-1} \rho^{m-1} + \mathcal{H}^{m-1}_g (\partial^* \Omega) = A (V)$. 
\end{enumerate}
\begin{proof} 
This is a combination of  \cite[Theorem 2.1]{Ritore-Rosales:2004} and \cite[Theorems 2.1 and 2.2]{Morgan-Ritore:2002}. Additional details for the special case where $(M, g)$ is asymptotically flat are given in \cite[Appendix E]{hdiso}. The idea is to consider a minimizing sequence $\{\Omega_i\}_{i=1}^\infty$ for the isoperimetric problem of volume $V$. Let 
\[
V_0 = \lim_{r \to \infty} \liminf_{i \to \infty} \mathcal{L}^m_g(\Omega_i \cap B_r).
\]
Standard arguments from geometric measure theory show that there is an isoperimetric region $\Omega$ in $(M, g)$ of volume $V_0$. Let $\rho \geq 0$ be such that 
\[
\omega_{m-1} \rho^m / m = V- V_0.
\]
Using cut and paste arguments we may replace the original minimizing sequence for volume $V$ by a sequence of the form $\{\Omega \cup R_i\}_{i=1}^\infty$ where $R_i \subset M$ are regular regions of volume $\omega_{m-1} \rho^m / m$ which diverge to infinity. Using \cite[Theorems 2.1 and 2.2]{Morgan-Ritore:2002} and a scaling argument, we conclude that
\[
\lim_{i \to \infty} \mathcal {H}^{m-1}_g (\partial R_i) = \omega_{m-1} \rho^{m-1}.\qedhere
\]  
\end{proof}
\end{lemma}


\begin{lemma}[\protect{\cite[Corollary 3.9]{Morgan-Ritore:2002}}]\label{ISOincone}
Assume that the link $(L, g_L)$ satisfies conditions \eqref{Riccicondition} and \eqref{areacondition}. For every $r > 0$ the set $(0,r)\times L$ uniquely minimizes perimeter for its volume in the cone $(C, g_C)$.
\end{lemma}


In \cite{Barbosa-doCarmo:1984}, J. Barbosa and M. do Carmo have classified closed volume preserving stable CMC immersions in Euclidean space; they are spheres. The following lemma describes an extension of their result to the case where the target manifold is a cone with non-negative Ricci curvature. It is due to F. Morgan and M. Ritor\'e.  

\begin{lemma}[Essentially \protect{\cite[Theorem 3.6]{Morgan-Ritore:2002}}]\label{CMCincone}
Let $(L, g_L)$ be a closed Riemannian manifold of dimension $m-1$ such that 
\begin{align} \label{lem:Ricciaux}
\Ric_L \geq (m-2) g_L
\end{align}
with strict inequality on a dense set. Let 
\[
C =  (0, \infty) \times L \qquad \text{ and } \qquad g_C = dr^2 + r^2 g_L
\]
be the cone on $(L, g_L)$. Consider an immersion of a connected manifold $S$ of dimension $m-1$ into $(C,g_{C})$ that has the following properties.  
\begin{enumerate} [(i)]
\item The immersion is complete away from the tip of the cone.
\item The immersion has finite area.  
\item The immersion is two-sided.
\item The immersion has constant non-zero mean curvature. 
\item The immersion is volume preserving stable in the sense that 
\[
\int_{S} |\nabla f|^2 \geq \int_{S} (|h|^2 + \Ric (\nu, \nu)) f^2
\]
for all $f \in C^\infty_c (S)$ with 
\[
\int_S f = 0.
\] 
Here, $\nu$ is a unit normal field along the immersion.
\end{enumerate}
Then the immersion is a covering of $\{r\}\times L$ for some $r > 0$.

\begin{proof}

The proof is almost exactly the same as for embedded surfaces in \cite[Theorem 3.6]{Morgan-Ritore:2002}. One uses volume preserving variations obtained from unit normal perturbation and homothetic rescaling to conclude from stability that the immersion is totally umbilical and that $\Ric (\nu, \nu) = 0$. The ``logarithmic cut-off trick" is used to deal with the cone point. For the characterization of totally umbilical immersions in $(C, g_C)$, we present an argument slightly different from \cite[Lemma 3.8]{Morgan-Ritore:2002}. 
The discussion in Appendix \ref{sec:geometryofcones} shows that the normal direction is radial at points along the immersion where \eqref{lem:Ricciaux} is strict. Such points lie dense in the open subset of $S$ where the immersion is transverse to the radial direction. Finally, note that the immersion cannot be everywhere tangent to the radial direction since this would entail that its trace contains an entire ray, implying that the immersion has infinite area by the monotonicity formula. 
\end{proof}
\end{lemma}


\begin{lemma} [Essentially \protect{\cite[Theorem 18]{Ros:2005}}] \label{lem:smallstableCMC} Let $(M, g)$ be a homogeneously regular $3$-dimensional Riemannian manifold. Every immersed complete volume preserving stable CMC surface of sufficiently large mean curvature is a perturbation of a geodesic sphere. 
\end{lemma}


\begin{lemma} [Essentially \protect {\cite[Proposition 2.2]{stablePMT}}] \label{largesmallmc} Let $(M, g)$ be a homogeneously regular $3$-dimensional Riemannian manifold. For every $\alpha > 0$ there is $\beta > 0$ such that the second fundamental form of every immersed complete volume preserving stable CMC surface whose mean curvature is bounded in absolute value by $\alpha$ is bounded in norm by $\beta$. 
\end{lemma}


\begin{rmk} The thresholds and estimates in Lemmas \ref{lem:smallstableCMC} and \ref{largesmallmc} only depend on the injectivity radius and on curvature bounds for the ambient manifold $(M, g)$.
\end{rmk}


\begin{lemma}[Essentially \protect {\cite[Proposition 2.3]{stablePMT}}]\label{decay-sff}
Let $(M, g)$ be a $3$-dimensional asymptotically conical Riemannian manifold such that the expansion (\ref{expansionmetric}) holds in $C^2$. For every $\alpha > 0$ and $r_0 > 0$ there is $\beta > 0$ such that for every connected complete volume preserving stable CMC immersion $\varphi : \Sigma \to (M, g)$ whose mean curvature is bounded in absolute value by $\alpha$ and whose trace intersects $B_{r_0}$ we have the estimate 
\[
\sup_{x \in \Sigma} |\varphi (x)||h(x)|\leq \beta.
\]
\end{lemma}


\begin{lemma} [\protect{\cite[(4)]{Christodoulou-Yau:1988}}] \label{lem:CY} Let $(M, g)$ be a $3$-dimensional Riemannian manifold. Then 
\begin{align} \label{CY}
\int_\Sigma H^2 + 2|h|^2 + 2 (R \circ \varphi)  \leq  64 \pi
\end{align}
for every connected closed volume preserving stable CMC immersion $\varphi : \Sigma \to (M, g)$. Here, $R$ is the scalar curvature of $(M, g)$ and $H$ and $h$ are respectively the mean curvature and second fundamental form of the immersion. The bound on the right hand side may be lowered to $48 \pi$ when the genus of $\Sigma$ is zero.
\end{lemma} 


\begin{lemma} [Essentially  \protect{\cite[(5.6), (5.7)]{Huisken-Yau:1996}}] \label{HawkingHY} Let $(M, g)$ be a $3$-dimensional asymptotically conical Riemannian manifold such that the expansion (\ref{expansionmetric}) holds in $C^1$ and such that 
\[
R \geq - O (r^{- 2 - \epsilon})
\]
for some $\epsilon > 0$ where $R$ is the scalar curvature of $(M, g)$. There exists a constant $\gamma > 0$ such that
\[
\int_\Sigma |h|^2 \leq  \gamma
\]
for every connected closed volume preserving stable CMC surface immersion $\varphi : \Sigma \to (M, g)$ such that $\underline{r}(\varphi (\Sigma)) > 1$ is sufficiently large. 

\begin{rmk} The assumption on the scalar curvature is satisfied when $K_L > 1$ and \eqref{expansionmetric} holds in $C^{2}$. Recall that $K_L$ is the Gaussian curvature of the link. It is also satisfied when $K_L \geq 1$ and when
\[
g = g_C + o(r^{-\epsilon})
\] 
in $C^2$ as $r \to \infty$. 
\end{rmk}

\begin{rmk} 
The arguments in \cite{stablePMT} show that  
\[
\overline {r} (\Sigma) \to \infty \qquad \text{ as } \qquad \text{area} (\Sigma) \to \infty
\]
when the scalar curvature of $(M, g)$ is positive. 
\end{rmk}

\begin{proof} 
The vector field 
\[
r \partial_r \in \mathfrak{X}(C)
\]
in $(C, g_C)$ shares all the properties of the position vector field of Euclidean space that are needed in the derivation of (5.6) in \cite{Huisken-Yau:1996}. The idea is then to combine (\ref{CY}) with the estimate 
\[
\underline{r} (\Sigma)^\epsilon \int_\Sigma r^{-2 - \epsilon} \lesssim \int_\Sigma H^2
\] 
valid provided $\underline{r}(\Sigma) > 1$ is sufficiently large. 
\end{proof}
\end{lemma}


\begin{lemma} [Essentially \cite{Gulliver-Lawson} and \cite{Fischer-Colbrie-Schoen:1980}] \label{lem:GulliverLawson} Let $(L, g_L)$ be a Riemannian surface with 
\[
K_L \geq 1.
\]
Let $(C, g_C)$ be the $3$-dimensional Riemannian cone on $(L, g_L)$. Thus
\[
C = (0, \infty) \times L \qquad \text{ and } \qquad g_C = dr^2 + r^2 g_L.
\]
Consider a stable minimal immersion of a two-dimensional manifold $S$ into $(C, g_C)$ that is complete away from the tip of the cone. Then $S$ diffeomorphic to either the cylinder or the plane. Moreover, the immersion is totally geodesic and 
\[
\Ric (\nu, \nu) = 0
\]
where $\Ric$ is the Ricci tensor of $(C, g_C)$ and $\nu$ is a unit normal field along the immersion. 
\begin{proof}
The condition on the Gaussian curvature of the link $(L, g_L)$ implies that 
\[
\Ric \geq 0.
\] 
Following an idea of R. Gulliver and B. Lawson \cite{Gulliver-Lawson}, we consider the conformally related cylindrical metric
\[
\tilde g_C = r^{-2} g_C = (d \log r)^2 + g_C.
\]
The scalar curvature of $\tilde g_C$ is positive. Moreover, the immersion is complete with respect to this metric. A computation using the stability of the immersion with respect to $g_C$, the Gauss equation, and the identity 
\[
\nabla (r \partial_r) = \text{id}
\]
in $(C, g_C)$ leads to the estimate
\[
- \tilde \Delta - \tilde K \geq 0
\]
where $\tilde \Delta$ and $\tilde K$ are the (non-positive) Laplace-Beltrami operator and the Gaussian curvature of the immersion with respect to $\tilde g_C$. The results of D. Fischer-Colbrie and R. Schoen \cite[Theorem 3]{Fischer-Colbrie-Schoen:1980} imply that $S$ is conformally equivalent to either the cylinder or the plane. Either way, as in \cite{Fischer-Colbrie-Schoen:1980}, we may find a sequence $f_{j} \in C_c^\infty(S)$ that converges to $1$ pointwise as $j \to \infty$ while their Dirichlet energies tend to $0$. Using these test functions in the stability inequality for the original immersion and passing to the limit, we obtain that 
\begin{equation*}
\int_{S} |h|^{2} + \Ric(\nu,\nu) \leq 0
\end{equation*}
from which the claim follows.
\end{proof}
\end{lemma}

The following result is an immediate consequence of the characterization of Ricci curvature bounds via optimal transport, cf. \cite{LottVillani,Sturm1,Sturm2}. We include an elementary and essentially well-known proof for convenient reference. We are grateful to Professors Anton Petrunin and Guofang Wei for valuable discussions related to this result. 

\begin{lemma} \label{lem:Riccilimit} Let $g, g_1, g_2, \ldots$ be Riemannian metrics on a manifold $M$ such that $g_i \to g$ in $C^0$. If each $g_i$ has non-negative Ricci curvature, then so does $g$. 
\begin{proof} Fix $p \in M$. By the Laplacian comparison theorem, 
\[
\Delta_{g_i} \, \text{dist}_{g_i} (p, \cdot) \leq \frac{m-1}{\text{dist}_{g_i} (p, \cdot)}
\]
holds weakly on $M \setminus \{p\}$ where $m$ is the dimension of $M$. These inequalities pass to the limit 
\begin{align} \label{auxupperbound}
\Delta_g \, \text{dist}_{g} (p, \cdot) \leq \frac{m-1}{\text{dist}_{g} (p, \cdot)}
\end{align}
as $i \to \infty$. Conversely, recall that $\Delta_g \, \text{dist}_{g} (p, \cdot)$ is the mean curvature of the geodesic sphere $S_r(p)$ for $r > 0$ sufficiently small. The expansion of this mean curvature at $q = \exp r \, \theta \in S_r(p)$ where $\theta$ is a unit tangent vector at $p$ is given by
\[
\frac{m-1}{r} - \frac{r}{3} \Ric (\theta, \theta) + O (r^2)
\]
as $r \searrow 0$, cf. \cite[p. 211]{Schoen-Yau:1994}. It follows that $\Ric \geq 0$. 
\end{proof}
\end{lemma}


\section{Proofs }

\begin{proof}[Proof of Theorem \ref{existencefoliation}] We consider the continuously differentiable map that takes a (small) function $u \in C^{2, \alpha}(L)$ to the mean curvature 
\[
H(g, u) \in C^{0, \alpha}(L)
\]
of the graph 
\[
\{ (1 + u(x), x) : x \in L \} \subset (0, \infty) \times L
\]
with respect to a $C^{2, \alpha}$ metric $g$ on $(0, \infty) \times L$. The (partial) linearization of this operator at the metric $g = dr^2 + g_L$ and at $u = 0$ with respect to $u$ is the linear map $C^{2, \alpha} (L) \to C^{0, \alpha} (L)$ given by 
\[
v \mapsto - \Delta_L v - (m-1) v.
\]
This operator is invertible. In fact, its lowest eigenvalue is $- (m-1)$ with eigenspace given by the constant functions. The next eigenvalue is 
\[
\lambda_1 (- \Delta_L) - (m -1) > 0.
\]
By the implicit function theorem, there is $\sigma > 0$ so that for all $g$  close to $g_C$ and constant functions $H$ close to $(m-1)$, the prescribed mean curvature equation 
\begin{align} \label{aux:pmc}
H (g, u) = H
\end{align}
has a unique solution $u \in C^{2, \alpha} (L)$ with $||u||_{2, \alpha} < \sigma$. Let $H_1, H_2$ be two constant functions with $H_2 < H_1$ that are close to $(m-1)$ and let $u_1, u_2$ be the corresponding solutions of (\ref{aux:pmc}). Standard analysis of the linear equation satisfied by their difference $u_2 - u_1$ gives that 
\[
u_1 < u_2.
\]
Conversely, by the maximum principle and standard estimates (e.g., \cite[Corollary 16.7]{Gilbarg-Trudinger:1998}), there is $\delta > 0$ such that every solution $u \in C^{2, \alpha} (L)$ of 
\[
H (g, u) = \text{constant}
\]
with $|u| < \delta$ and $g$ sufficiently close to $g_C$ satisfies 
\[
||u||_{2, \alpha} < \sigma.
\]
The theorem follows from this and a scaling argument.
\end{proof}


\begin{proof}[Proof of Theorem \ref{thm:uniquenessiso}] Let $V_i \to \infty$. Using Lemma \ref{lem:minimizingsequences} we find $\rho_i > 0$ and $\Omega_i \subset M$ isoperimetric such that 
\[
\omega_{m-1} \rho_i^m / m  + \mathcal{L}^m_g (\Omega_i) = V_i
\]
and 
\[
\omega_{m-1} \rho_i^{m-1} + \mathcal{H}^{m-1}_g (\partial^* \Omega_i) = A (V_i).
\]
We may think of $\Omega_i \setminus B_2$ as a region in $(1, \infty) \times L \subset C$. We take the union of $\Omega_i$ with $B_2$ and a far out geodesic ball of volume $\omega_{m-1} \rho_i^m / m$ to account for the stray volume. Upon scaling down homothetically by a factor close to 
\[
r_i = ( (m V_i)/\text{area} (L, g_L) )^{1/m}
\]
we obtain a minimizing sequence $\tilde \Omega_i$ in $(C, g_C)$ for the isoperimetric problem of volume 
\[
\text{area}(L, g_L) / m.
\] 
By Lemma \ref{ISOincone}, the unique solution of this problem is given by $(0, 1) \times L$. It follows that the stray volumes $\omega_{m-1} \rho_i^m / m$ are negligible when compared to $V_i$ as $i \to \infty$, i.e. that 
\begin{align} \label{aux:scales}
\rho_i \ll r_i.
\end{align}
Standard compactness arguments for isoperimetric boundaries show that the reduced boundaries of the regions $\tilde \Omega_i$ converge to the cross-section 
\[
\{1\} \times L
\]
locally as measures in $C$. Allard's theorem shows that this convergence is in fact in $C^{1, \alpha}$. A hole filling argument as in \cite[p.\ 175]{isostructure} gives that $\tilde \Omega_i$ has no other boundary than that near $\{1\} \times L$ for $i$ sufficiently large. These results can now be lifted to $\Omega_i$. Finally, if $\rho_i > 0$ for large $i$, a first variation type argument shows that the mean curvature of $\partial \Omega_i \sim (m-1) / r_i$ is equal to $(m-1) / \rho_i$. This contradicts (\ref{aux:scales}). 
\end{proof}

\begin{proof}[Proof of Theorem \ref{thm:uniquestable}]
Let $\Sigma_k$ be a connected closed volume preserving stable CMC surface in $(M, g)$ with $\underline{r} (\Sigma_k) \geq k$. We homothetically rescale 
\[
\Sigma_k \setminus \bar B_2 \qquad \text{ to } \qquad S_k
\]
and 
\[
(M \setminus \bar B_2 \cong (2, \infty) \times L, g) \qquad \text{to} \qquad (M_k \cong (2/\overline{r} (\Sigma_k), \infty) \times L, g_k)
\]
so that 
\[
\overline{r} (S_k) = 1
\]
in $(M_k, g_k)$.  Note that 
\[
(M_k, g_k) \to (C, g_C)
\] 
locally in $C^{2, \alpha}$. The maximum principle shows that 
\[
H_k \geq 2 - o(1)
\]
as $k \to \infty$ where $H_k$ is the mean curvature of $S_k$ in $(M_k, g_k)$. The estimate
\begin{align} \label{auxhestimate}
\int_{S_k} |h_k|^2 \leq \gamma
\end{align}
follows from scaling and Lemma \ref{HawkingHY}. Using also the lower bound on the mean curvature, we see that the area of $S_k$ is a priori bounded. There is $\beta > 0$ depending only on $(C, g_C)$ such that if $H_k > \beta$ and $k$ is sufficiently large, then $S_k$ is close to a geodesic coordinate sphere in $(C, g_C)$ by Lemma \ref{lem:smallstableCMC} (or rather its proof). Assume that $H_k \leq \beta$ for all large $k$. By Lemma \ref{decay-sff}, we have bounds
\[
 \sup_{k \geq 1} \sup_{x \in S_k} |x| |h_k (x)| < \infty
\] 
for the second fundamental form. Using the geometric version of the Arzel\`a-Ascoli theorem, we may extract a subsequential limit immersion $S$ of $S_k$ in $(C, g_C)$ that satisfies the conditions of Lemma \ref{CMCincone}. The trace of this immersion is equal to $\{1\} \times L$. It follows that $\Sigma_k$ is close in scale to a centered sphere in $(M, g)$ for all $k$ sufficiently large. The assertion follows from the uniqueness part of Theorem \ref{existencefoliation}.
\end{proof}


\begin{proof}[Proof of Theorem \ref{thm:min-surf}] We argue by contradiction. Let $\Sigma$ be as in the statement of the theorem. We may assume it is connected. As in the proof of Theorem \ref{thm:uniquestable} above, we homothetically rescale with respect to points whose images in $M$ diverge and then pass to a non-trivial stable minimal limiting immersion $\varphi : S \to (C, g_C)$ that is complete away from the tip of the cone. Lemma \ref{lem:GulliverLawson} implies that the Ricci tensor of $(C, g_C)$ vanishes in the normal direction along this immersion. Using this, our assumption that $K_L > 1$, and the discussion in Appendix \ref{sec:geometryofcones}, we see $\nu$ is radial. Hence the trace of the limiting immersion must be a slice in the cone. This contradicts the minimality of the limiting immersion.
\end{proof}

\begin{proof} [Proof of Theorem \ref{thm:coneangleestimate}] By Lemma \ref{lem:Riccilimit}, the asymptotic cone $(C, g_{C})$ of $(M, g)$ has non-negative Ricci curvature. This is equivalent to the estimate $\Ric_{L} \geq (m-2)g_{L}$ for the link $(L, g_L)$ of $(C, g_C)$. By Bishop's theorem, $\area{(L, g_L)} \leq \omega_{m-1}$ with equality only for the unit sphere. The characterization of isoperimetric regions of $(C, g_C)$ in Lemma \ref{ISOincone} gives that
\[
c_{iso}(C,g_C) =  \area(L,g_L) / \omega_{m-1}.
\]
In fact, the infimum in \eqref{def:coneangle} is achieved by slabs $(0, r) \times L$ for every $r > 0$. Together with a comparison argument for slabs of large volume, we obtain that $c_{iso} (M, g) \leq c_{iso}(C, g_C)$. Assume now that $c_{iso} (M, g) = 1$. It follows that $(L, g_L)$ is the unit sphere so that $(M, g)$ is asymptotically flat. A standard application of Bishop's theorem gives that $(M, g)$ is isometric to Euclidean space.
\end{proof}


\appendix 


\section{Geometry of Cones} \label{sec:geometryofcones}

Let $(L, g_L)$ be an $m-1$ dimensional Riemannian manifold. We consider the Riemannian cone $(C, g_C)$ where 
\[
C = (0, \infty) \times L \qquad \text{ and } \qquad g_C = dr^2 + r^2 g_C.
\]
For every $r > 0$, the slice 
\[
\{r\} \times L
\]
is umbilical with constant mean curvature $(m-1)/r$.  We have that 
\[
\nabla_Y (r \partial_r) = Y 
\]
for all $Y \in \mathfrak{X}(C)$ where $\nabla$ is the Levi-Civita connection of $(C, g_C)$. We have that 
\begin{align*}
\Ric_C (\partial_r, \partial_r) &= 0 \\
\Ric_C(\partial_r, X) &= 0 \\
\Ric_C (X, X) &= (\Ric_L - (m-2) g_L )(X, X) 
\end{align*}
for all $X \in \mathfrak{X}(C)$ that are tangent to $L$. The condition
\begin{align} \label{auxapp}
\Ric_L \geq (m-2) g_L
\end{align}
on the link implies that $(C, g_C)$ has non-negative Ricci curvature, and conversely. At points where \eqref{auxapp} is strict, the radial direction is the unique eigendirection of the Ricci tensor of $(C, g_C)$ with vanishing eigenvalue.


\section{A remark on condition \eqref{Riccicondition}}  \label{app:alternativecondition}

In this section, we observe that Theorem \ref{thm:uniquenessiso} holds in a much broader class of cones than just those whose link $(L, g_L)$ satisfies \eqref{Riccicondition} and \eqref{areacondition}. To this end, we define the isoperimetric profile of a closed $m-1$ dimensional Riemannian manifold $(L, g_L)$ as follows. Given $\beta \in (0,1)$, we let
\begin{equation*}
\cI_{(L,g_{L})}(\beta) = \inf\left\{\frac{\cH^{m-2}_{g_{L}}(\partial^{*}\Omega)}{\cL^{m-1}_{g_{L}}(L)} : \text{$\Omega$ open with finite perimeter and } \cL^{m-1}_{g_{L}}(\Omega) = \beta \cL^{m-1}_{g_{L}}(L)  \right\}.
\end{equation*}
Note that $\cI_{(L,g_{L})}$ is symmetric with respect to $1/2$. It is convenient to set $\cI_{(L,g_{L})}(0) = \cI_{(L,g_{L})}(1) = 0$. 

\begin{lemma}[Levy--Gromov {\cite[p. 520]{Gromov:1999}}] Let $(L, g_L)$ be a closed Riemannian manifold of dimension $m-1$.
If 
\[
\Ric_{L} \geq (m-2)g_{L}
\]
and $(L,g_{L})$ is not the unit sphere, then
\[
\cI_{(L,g_{L})}(\beta) > \cI_{(\mathbb{S}^{m-1},g_{\mathbb{S}^{m-1}})}(\beta)
\]
for all $\beta \in (0,1)$.
\end{lemma}

As such, the condition 
\begin{equation}\label{Isocondition}
\cI_{(L,g_{L})}(\beta) > \cI_{(\mathbb{S}^{m-1},g_{\mathbb{S}^{m-1}})}(\beta) \qquad \text{for all} \qquad \beta \in (0,1)
\end{equation}
is a generalization of condition \eqref{Riccicondition}.

Note that  
\[
\cI_{(L,\rho^{2} g_{L})}(\beta) = \rho^{-1}\cI_{(L, g_{L})}(\beta),
\]
for every $\rho > 0$. Recall that up to fudge factors, the Euclidean isoperimetric inequality holds for small volumes in every homogeneously regular Riemannian manifold. It follows that there exists $\rho_{0} = \rho_{0}(L)>0$ so that 
\[
(L,\rho^{2} g_{L})
\]
satisfies \eqref{Isocondition} for all $\rho \in (0, \rho_{0})$. The analogous statement for conditions \eqref{Riccicondition} and \eqref{areacondition} clearly fails. An estimate for the largest possible $\rho_{0}$ is given in terms of a (possibly negative) lower bound for the Ricci curvature and the diameter of $(L,g_{L})$ in \cite{Berard-Besson-Gallot:1985}.

In the following lemma, we record the observation of F. Morgan \cite[Section 3.2]{Morgan:2007} that the isoperimetric product theorem {\cite[Proposition 8]{Ros:2005}} of A. Ros also holds in warped products. 

\begin{lemma} [\cite{Morgan:2007, Ros:2005}]
Suppose that \eqref{areacondition} and \eqref{Isocondition} hold and let $r > 0$. Then $(0, r) \times L$ is the unique isoperimetric region of its volume in the cone $(C, g_C)$ with link $(L, g_L)$.
\end{lemma}

The following generalization of the Lichnerowicz eigenvalue estimate is due to P.\ B\'erard and D.\ Meyer \cite{Berard-Meyer:1982}. 

\begin{lemma} [\cite{Berard-Meyer:1982}]
Suppose that \eqref{areacondition} and \eqref{Isocondition} hold. Then $\lambda_{1}(L,g_{L}) > m-1$.
\end{lemma}

From this, it is easy to see that condition \eqref{Riccicondition} may be replaced by \eqref{Isocondition} in Theorem \ref{thm:uniquenessiso}.

\bibliographystyle{amsplain}
\bibliography{references}

\end{document}